\documentclass[11pt]{article}
\usepackage{rotating,multirow,amsmath,amssymb,fullpage}

\usepackage{natbib}
\usepackage{setspace}
\usepackage{color}
\usepackage{xcolor}
\bibpunct[, ]{(}{)}{,}{a}{}{,}%

\newtheorem{definition}{Definition}
\newtheorem{theorem}{Theorem}
\newtheorem{lemma}{Lemma}

\newtheorem{corollary}{Corollary}

\newenvironment{proof}[1][Proof]{\textbf{#1.} }{\ \rule{0.5em}{0.5em}}
\title{Search and Delivery Man Problems: \\ When Are Depth-First Paths Optimal?}

\author{Steve Alpern\thanks{Warwick Business School, University of Warwick, Coventry CV4 7AL, United Kingdom, steve.alpern@wbs.ac.uk} \and Thomas Lidbetter\thanks{Department of Management Science and Information Systems, Rutgers Business School, Newark, NJ 07102, tlidbetter@business.rutgers.edu (corresponding author)}}

\providecommand{\keywords}[1]{\textbf{\textbf{Keywords:}} #1}

\begin{document}

\maketitle

\begin{abstract} \noindent Let $h$ be a probability measure on the nodes and arcs of a network $Q$, viewed either as the location of a hidden object to be found or as the continuous distribution of customers receiving packages. We wish to find a trajectory starting from a specified root, or depot $O$ that minimizes the expected search or delivery time. We call such a trajectory optimal. When $Q$ is a tree, we ask for which $h$ there is an optimal trajectory that is depth-first, and we find sufficient conditions and in some cases necessary and sufficient conditions on $h$. A consequence of our analysis is a determination of the optimal depot location in the Delivery Man Problem, correcting an error in the literature. We concentrate mainly on the search problem, with the Delivery Man Problem arising as a special case. 
\end{abstract}



\newpage 

\section{Introduction} \label{sec:intro}

A person, or more generally a {\em target}, is lost in a network. We know the lengths of the arcs of the network, and we have a knowledge of the probability distribution (or {\em hiding distribution}) according to which the target is hidden. We wish to choose a unit speed trajectory (or {\em search}) starting from a given point (the {\em root}) with the aim of minimizing the expected {\em search time}: that is the expected time to find the target. The target might be a child lost in a cave, as in the recent rescue in Thailand; or it might be a leak in a pipeline. The Searcher could be an underwater diver or an unmanned aerial vehicle (UAV), respectively. We call a search that minimizes the expected search time {\em optimal}. The target does not have to be at a node; it can be anywhere in the interior of any arc. This paper considers the problem when the network is a tree.

The hiding distribution can be interpreted as a distribution of mail recipients along the roads, in which case we seek the delivery route which minimizes the mean time for a package to be received. This generalizes the case of one recipient at each node, where it is called the  {\em Delivery Man Problem} or {\em Traveling Repairman Problem}, and the case of the uniform hiding distribution (defined formally in Section~\ref{sec:GEBD}), where the probability that the target belongs to a region (measurable subset) of the network is proportional to the total length of that region. The latter problem was coined the {\em Utilitarian Postman Problem} by~\cite{Alpern07}, and some general results on this problem in~\cite{ABG}.

One method of searching a tree is what is known as {\em depth-first (DF)}
search, which follows a sequences of arcs, starting and ending at the root,
where the unique arc going back to the root is chosen only after all other
arcs from the node have been traversed. The main problem addressed in this
paper is the titled question: for which hiding distributions is a DF search
optimal? We call such hiding distributions {\em simply searchable}. 

The recent work of~\cite{Li-Huang} shows that, under a certain assumption on the optimal search, the uniform hiding distribution
on a tree is simply searchable. Earlier it was shown by \cite{Kella} that the uniform distribution on a star is simply searchable. \cite{Beck-Beck} and~\cite{Baston-Beck} also proved this for a finite interval with interior start, with the more general objectives of minimizing certain functions of the expected search time. 

Another distribution on a tree that has been shown to be simply searchable
is the distribution shown by \cite{Gal79} to be optimal for a time maximizing
hider in what is known as a search game. To describe Gal's distribution, now
known as the {\em Equal Branch Density (EBD)} distribution, we must define the
elementary concept of {\em search density}. The search density of a region of a network is the ratio of the probability the target is located in the region and the length of the region.  For example, the uniform distribution can be characterized as the unique hiding distribution for which the search densities of all regions are equal. If a point $x$ is removed from a tree, we call the components not containing the root in the resulting network the {\em branches} at $x$. A hiding distribution on a tree is called {\em balanced} if the densities of all the branches at $x$ are equal, for any point $x$ on the tree. Gal's EBD distribution is defined as the unique
balanced distribution on a tree whose support is the set of leaf nodes. See Section~\ref{sec:GEBD} for an example of the construction of the EBD distribution.
\cite{Alpern11} showed that the EBD distribution is simply searchable and \cite{AL14} showed that the optimal searches are exactly the DF searches. 

In this paper we define a class of distributions called {\em monotone}, encompassing the uniform
distribution and the EBD distribution, and we show they are all simply
searchable. We say that a hiding distribution is monotone if whenever $x$ is on the path from the root to $y$, the density of the set of branches at $x$ is at most that of the set of branches at $y$. Roughly speaking, this says that the
target is more likely to be far than close to the root along any path from
root to leaf node. Theorem~\ref{thm:monotone} says that monotone distributions are simply searchable. 

The important paper of~\cite{Kella} considered search on star
networks. For such networks we say that a hiding distribution is {\em forward biased} if the probability the target is located within distance $x$ of the root on a given arc is bounded above by a certain function $H(x)$, given explicitly in Section~\ref{sec:star}. Such an
upper bound clearly limits the probability that the target is close to the
center. For general star networks, Theorem~\ref{thm:star} says that a balanced
distribution is simply searchable if and only if it is forward-biased.
For two-arc stars (intervals) we can remove the assumption of balanced.
Theorem~\ref{thm:FB} says that for intervals a hiding distribution is simply searchable if and only if it
is forward-biased. The connection between our results and those of Kella for
the star are discussed in Subsection~\ref{sec:kella}. In particular, we give a hiding distribution on an interval that is forward biased and balanced, hence simply searchable by Theorem~\ref{thm:star} or~\ref{thm:FB}, but does not satisfy Kella's condition for simply
searchability.

One reason for determining sufficient conditions for simple searchability is
that for simply searchable balanced distributions this gives us a simple
formula for the minimum expected search time, given in Theorem~\ref{thm:GEBD}.  This allows us to show that for both the Utilitarian Postman Problem and the Delivery Man Problem, the optimal depot location (starting point of the search) is the leaf node of minimal closeness centrality (the one which maximizes the sum of distances to all nodes). Our result for the Delivery Man Problem corrects an error from~\cite{Minieka}.

The paper is organized as follows. Section~\ref{sec:lit} is a literature review. Section~\ref{sec:def} presents the main definitions for the paper and in particular describes a new way of partitioning a tree into two parts with respect to a given search: a rooted subtree where the search fails to be DF; and the rest, where
it searches the complementary trees in a DF manner. Section~\ref{sec:GEBD}
proves the formula stated above for the minimal search time for a simply
searchable balanced distribution and gives consequences for the uniform
distribution and the case where the starting point (root) is a choice
variable. Section~\ref{sec:monotone} proves Theorem~\ref{thm:monotone}, described above, that monotone distributions are simply
searchable. Section~\ref{sec:depot} solves the problem of optimal depot location on a tree for the Delivery Man Problem. Section~\ref{sec:star} gives our results for star networks, that
forward biased is a necessary and sufficient condition for simple
searchability. Subsection~\ref{sec:kella} compares our results for stars with those of~\cite{Kella}. Section~\ref{sec:conclusion} concludes.

\section{Literature Review} \label{sec:lit}

There is a considerable literature on the so-called {\em Linear Search Problem}, originating with \cite{Beck} and \cite{Franck}, which seeks an optimal search for a target located on the real line according to a known distribution. This has been extended by \cite{Kella} to search on a star, where it was shown that against a uniform hiding distribution on a star, all DF searches are optimal. Kella's results are discussed more fully in Subsection~\ref{sec:kella}.

The problem of finding an optimal search path for a target located uniformly on an interval is considered in \cite{Beck-Beck} and \cite{Baston-Beck}. Both these works considered a more general objective than the search time $T$: the former considers the objective $T^\alpha$ for some $\alpha>1$; the latter considers a general convex function of $T$. Both papers show that DF searches are optimal when the root is taken as any point in the interval.

More recently, \cite{Li-Huang} address the problem of finding an optimal search in the case that the target is hidden on a network according to the uniform distribution, as discussed in Section~\ref{sec:intro}. \cite{Li-Huang} make the assumption that the Searcher never turns in the interior of an arc. They prove that under this assumption any DF search of a tree is optimal for the uniform distribution.

\cite{Li-Huang} also show that for the uniform hiding distribution on a tree, if the Searcher can choose her starting point, then she should choose a leaf node. However, they do not give a way to determine {\em which} leaf node should be chosen, and we fill this gap in this paper. Search games in which the Searcher chooses her starting point have been studied in \cite{DG08} and \cite{ABG08}. More recently, \cite{Alpern18} studied a search game in which the Searcher can choose her starting point from a given subset of the network.

If the target is confined to the nodes of the network, the problem of finding the optimal search is a discrete one.  In the case that the target is on each node with equal probability, the problem of finding the optimal search is equivalent to minimizing the {\em latency} of a network, also known as the {\em Delivery Man Problem} or {\em Traveling Repairman Problem} (see \cite{Blum}, \cite{Goemans}, \cite{Sitters} and \cite{Arora}). \cite{Minieka} showed that for a tree with unit arc lengths, a DF search is optimal for this problem. By placing nodes of equal probability at approximately equal intervals on all arcs, we can approximate a continuous uniform distribution and the result of \cite{Li-Huang}.  \cite{BK19} solve the problem of finding the optimal search for a target located according to known probabilities on the nodes of a bipartite network, when the nodes have search costs.

Our approach to searching a network in this paper is to assume the target is located according to some known probability distribution. This differs from the approach taken in the search games literature, of determining a randomized search of the network that minimizes the expected search time in the worst case, which can be equivalently framed as a zero-sum game between the target and the Searcher. That approach to searching a network has been studied extensively in the search games literature, for example in \cite{Gal79}, where this field of study was initiated; \cite{Garnaev}; \cite{AG03} and more recently in \cite{AL13}, \cite{AL14}, \cite{BK13} and \cite{Lin16}.

\section{Definitions} \label{sec:def}

We start by giving rigorous definitions of a search and a hiding distribution, and what it means for a search to be optimal. We then go on to define precisely DF search, and the notion of {\em search density}, which will be an important tool in our analysis.

\subsection{Searches and hiding distributions}
Let $Q$ be a tree network with root node $O$. The length (or Lebesgue measure) of an arc $a$ of $Q$ is denoted $\lambda(a)$, which extends naturally to a measure $\lambda$ on $Q$, with total measure denoted by $\mu =\lambda(Q)$. This defines a metric $d$ on $Q$ such that $d(x,y)$ is the length of the (shortest) path $P(x,y)$ from $x$ to $y$. Note that we model the network as continuous, in the sense that $x$ and $y$ can be points in the interior of arcs as well as nodes. 



The {\em branch nodes} of a tree $Q$ consist of the root $O$ and all other nodes of degree at least 3. Let $\preceq$ be the natural partial order on points of $Q$, so that $x \preceq y$ if $x$ is on the path between $O$ and $y$. For a point $x \in Q$, let $Q_x = \{y \in Q: x \preceq y\}$ be the subtree rooted at $x$. The connected components of $Q_x-\{x\}$ are called the {\em branches} at $x$, and clearly there are at least two branches at $x$ if and only if $x$ is a branch node.

We define a {\em search} of the tree to be a 
function $S: [0,\infty) \rightarrow Q$ satisfying $S(0)=O$ (starts at the root) and $d(S(t_1),S(t_2)) \le t_2-t_1$, for all times $0 \le t_1 < t_2$. That is, a search is a unit speed path on the metric network $Q$. Of course if $Q$ has been covered by some time $M$ (that is, $Q=S([0,M])$), then the behavior of $S$ after time $M$ is irrelevant. We denote the set of all searches by $\mathcal S$. We consider $\mathcal{S}$ with the topology of uniform convergence on compact sets.

The hiding distribution $h$ is a Borel probability measure on the network, viewed as a compact metric space $(Q,d)$. For a given search $S$ and a given point $x$ in $Q$, denote the time taken for $S$ to first reach $x$ (the {\em search time}) by $T(S,x) = \min \{t \ge 0: S(t)=x \} $, which we allow to be $+\infty$ (but not for reasonable covering searches). 
Note that $T(S,x)$ is lower semi-continuous in $S$. 

Similarly, we denote the {\em expected search time} by $T(S,h) = \int_{x \in Q} T(S,x) ~dh(x)$, which is also lower-semicontinuous in $S \in \mathcal S$. Since $\mathcal{S}$ is compact in this topology, it follows that there is a search $S$ that minimizes $T(S,h)$, for given $h$, and we refer to such a search as {\em optimal}. We write $V(h)$ for the expected search time $T(S,h)$ of an optimal search $S$ against a hiding distribution $h$. We know that the minimum expected search time cannot be more than $2\mu $ because the tree $Q$ can be searched in time $2\mu $. We summarize these discussions in the following theorem. The details are standard, having first being proved in Appendix 1 of \cite{Gal80}.
\begin{theorem} \label{thm:existence}
For any hiding distribution $h$, there exists some search $S \in \mathcal S$ that minimizes $T(S,h)$.
\end{theorem}

Recall that $P(x,y)$ denotes the shortest path between two points $x,y \in Q$. We distinguish searches with following property.

\begin{definition} \label{def:normal} A search $S$ is {\bf normal} if for any times $t_1,t_2$ with $S([t_1,t_2])=P(S(t_1),S(t_2))$, it is the case that at time $t_1$ the search $S$ goes directly from $S(t_1)$ to $S(t_2)$. More precisely, for $0 \le \theta \le 1$, the point $S(\theta t_1+(1-\theta)t_2)$ is the point on $P(S(t_1),S(t_2))$ that is at distance $\theta t_1+(1-\theta)t_2$ from $S(t_1)$.
\end{definition}
It is clear that against a fixed hiding distribution, there must be an optimal search that is normal, since a search that is not normal can be replaced by a normal search whose expected search time is no greater. We therefore assume for the rest of the paper that all searches considered are normal.

The following paragraph is only for background. When $Q$ is an infinite line, $h$ has bounded support $[ x^{-},x^{+}] $ and the search starts at $O=0$ (the Linear Search Problem) the search is typically described by as a {\em generalized search strategy}, given by a doubly infinite sequence $x=\{x_{i}\} _{i=-\infty }^{\infty }$  satisfying 
\begin{align}
x^{-} \leq \dots \leq x_{-i-1}\leq x_{-i}\leq \dots \leq 0\leq \cdots \leq x_{i}\leq x_{i+1}\leq \dots \leq x^+.   \label{sequence notation}
\end{align}
That is, the Searcher employs a path in which, for each integer $r$, he goes from $x_{r}$ to $x_{r+1}$.  \cite{Beck2} (Theorem 12) has shown that infinite oscillations at the start are not required if the cumulative distribution function (cdf) $F$ for the hiding distribution has a finite right or left derivative at the origin. \cite{Kella} noted that a similar result also holds for stars, and it is clear that this further extends to all trees. This condition on the cdf will be true for the hiding distributions we consider here, but in fact we do not need this result. When there is an integer $i$ such that $x_{i}=x^{-}$ or $x^{+}$ it is said that the strategy is {\em terminating}. If, for some $m$, the $x_{i}=0$ for all $i<m$, this is called a {\em standard search strategy}, and it starts with a first step from $0$ to $x_{m}$.  For some hiding distributions, the optimal search $S$ may not be terminating. For instance, \cite{Beck-Beck} showed this to be the case in the context of the Linear Search Problem for the triangular distribution on the interval $[-1,1]$, with probability density function (pdf) $f(x)=1-|x|$. \cite{Baston-Beck} (Theorem~5.2) have shown that it is sufficient to consider terminating search strategies if either
\[
\lim_{t\searrow x^{-}}\inf F\left( t\right) /\left( t-x^{-}\right) >0 \text{ or } \lim_{t\nearrow x^{+}}\inf F\left( t\right) /\left( x^{+}-t\right) >0.
\]
Later in this section we will prove in Theorem \ref{thm:leafy} an analogue of that result for trees. \cite{Kella} has adapted the sequence notation (\ref{sequence notation}) to star networks. We will not use these notations here.

\subsection{Depth-first search}
We are interested in this paper in when a DF search is optimal on a tree. We give the formal definition of DF below.

\begin{definition} 
A {\bf depth-first (DF)} search of a tree $Q$ with root $O$ is a sequence of arcs,
traversed at unit speed, starting and ending at $O$ such that, when leaving a node,
the unique arc towards the root is only chosen if all the other arcs have been
traversed. 
\end{definition}
Note that a DF search ends back at the root having traversed every arc once in each direction. Thus if $S$ is DF we have $S(0)=S(2\mu)=O$. In fact, any search $S$ with $S[0,2\mu] = Q$ is necessarily DF.

An example of a DF search on the tree depicted on the left in Figure~\ref{fig:network} is the one that visits the nodes in the order $O,A,O,D,C,D,B,D,O$. Given the indicated arc lengths, the search takes time $2(6+3+2+3)=28=2 \mu$.

\begin{figure}[h]
\begin{center}
\includegraphics[scale=0.5]{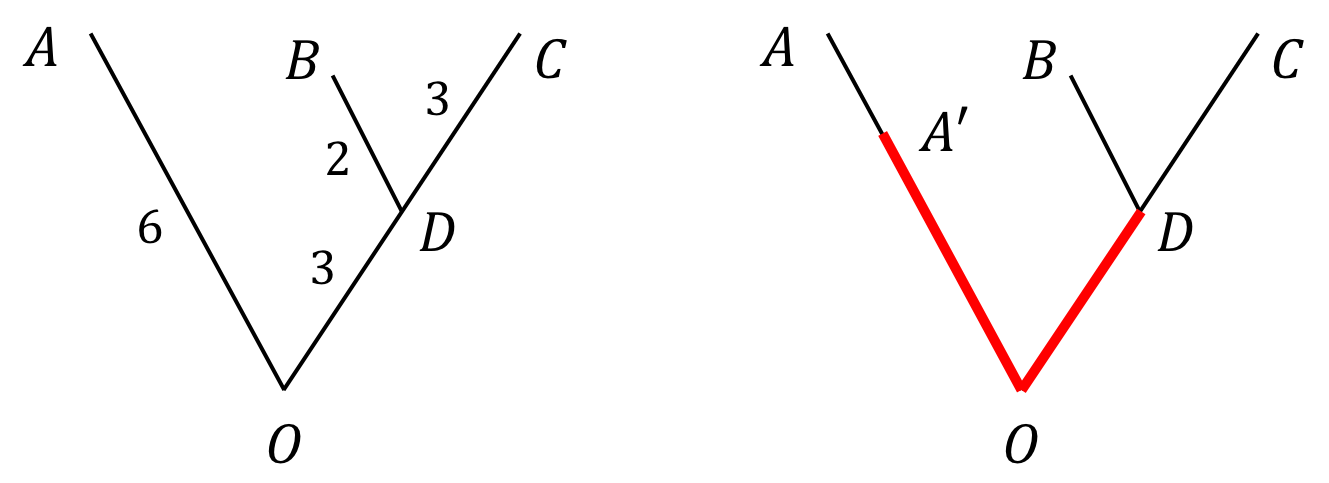}
\caption{A tree network}
\label{fig:network}
\end{center}
\end{figure}

\begin{definition} A search $S$ is terminating if $S([0,M])=Q$ for some $M$. Equivalently, $S$ must reach all leaf nodes.
\end{definition}
For any terminating search $S$ on a tree $Q$ with starting at the root $O$, consider the points $x \in Q$ for which the following condition holds.
\begin{align}
S([t_1, t_1+2\lambda(Q_x)])=Q_x, \text{ where } t_1 = T(S,x) \text{ and } S(t_1+2\lambda(Q_x)) = x. \label{eq:DFcond}
\end{align}
Condition~(\ref{eq:DFcond}) says that $S$ carries out a DF search of the subtree $Q_x$ upon reaching $x$ for the first time. This allows us to partition $Q$ into a set $D = D_S$ (the {\em DF set for $S$}) where the condition holds and a set $N = N_S$ (the {\em non-DF set for S}) where it does no't hold. Clearly $S$ is DF if and only if $D = Q$, or equivalently $N = \emptyset$.

As an example, consider the tree on the right of Figure~\ref{fig:network}, where $A'$ is located somewhere on the arc $OA$. Let $S$ be the search that visits the nodes in the order $O,A',O,D,B,D,O,A',A,A',O,D,C$. The part of the network highlighted by the thick red line is $N_S$ and the remaining part of the network is $D_S$. Note that $N_S$ is a closed subtree of the network, containing the root $O$ and none of the leaf nodes. This is true in general.

\begin{lemma} \label{lem:DFdecom}
If $S$ is a terminating search of a tree $Q$ starting at root $O$, then its non-DF set $N=N_S$ is a closed subtree of $Q$ containing $O$ (if it is non-empty) and none of the leaf nodes.
\end{lemma}
\begin{proof}
It is clear that if some point $x \in D=D_S$, and $x \preceq y$, then $y \in D$. It follows that if $Q$ is non-empty then $N$ is a subtree containing the root $O$. It is also clear that $D$ contains all the leaf nodes, since $S$ is terminating.

To show that $N$ is closed, we prove that $D$ is open. Let $x \in D$, and let $y \preceq x$ be the closest branch node below $x$, with $y \neq x$. Let $t_1$ be the latest time before $t_2=T(S,x)$ such that $S(t_1)=y$. Then we must have $S([t_1,t_2])=P(y,x)$ and since $S$ is normal, it must go directly from $y$ to $x$ at time $t_1$. Since $S([0,t_1])$ is closed and does not contain $x$, there must be an open ball around $x$ consisting of points that have not been visited by $S$ at time $t_1$. Clearly, all these points must belong to $D$, and $D$ is therefore open.
\end{proof}

\subsection{Search density}
An important notion we will use is that of {\em search density}.
Consider the restriction $S'$ of a search $S$ to some interval $[t_1,t_2]$ with $t_2>t_1$, and let $F$ be the cumulative capture probability of $S$ given by $F(t) = h(S([0,t]))$ for $t \le M$. Then the search density of $S'$ is defined as 
\[
\rho(S') = \frac{F(t_2) - F(t_1)}{t_2 - t_1}.
\]
In a slight abuse of notation, we will also refer to the search density of a region $R \subset Q$ with $\lambda(R)>0$ and denote it $\rho(R) = h(R)/\lambda(R)$.

We will need to use a theorem that extends Theorem 5.2 from \cite{Baston-Beck} that, roughly speaking, says that if the measure of the hiding distribution near the leaf nodes is concentrated enough, then there is a terminating optimal search. (For the purposes of this paper, we do not consider $O$ to be a leaf node even if it has degree $1$.)

First, we define the notion of a {\em leafy} hiding distribution to be one for which at all leaf nodes, or all but one leaf node $v$, we have
\begin{align}
\liminf_{x \rightarrow v} \rho(Q_x) > 0. \label{eq:leafy}
\end{align}
\begin{theorem} \label{thm:leafy}
Suppose $h$ is a leafy hiding distribution on a tree. 
Then any optimal strategy $S$ is terminating.
\end{theorem}
\begin{proof}
First suppose $v$ is a leaf node for which~(\ref{eq:leafy}) holds, and we will show that for some time $\tau>0$, we have $S(\tau) = v$. Suppose $v$ is on an arc of length $L$ whose other endpoint is $w$. We identify the arc with the interval $[0,L]$, where $x=0$ corresponds to $v$ and $x=L$ corresponds to $w$. Let $F(x) = h([0,x])$ be the measure of the set $Q_x \equiv [0,x]$ of points within distance $x$ of $v$. Then we have, by assumption, that
\[
\liminf_{x \rightarrow v} \frac{F(x)}{x} = m,
\]
for some $m>0$. It follows that for some $\delta$ with $ 0 < \delta < L/2$, 
\[
\frac{F(x)}{x} > \frac{m}{2} \text{ for } 0 < x \le \delta.
\]
Let $t$ be such that $h(S([0,t])) >1 - \varepsilon$, where $\varepsilon < Lm/4$. 

Now suppose it is not the case that $S(\tau) = v$ for some time $\tau$. In this case, there must be some $t'>t$ such that $S$ reaches some point at distance at most $\delta$ from $v$ at time $t'$. At some later time $t''$, the search $S$ must revisit $w$, by the normality condition. Let $a$ be the closest point to $v$ in $S([0,t''])$. Let $d(v,a)=\delta' \le \delta$ and let $t^* > t'$ be such that $S(t^*)=a$. Let $S'$ be the same as $S$ up until time $t^*$, whereupon it goes to $v$ and back to $a$, and then follows $S(t-2 \delta')$, the original path of $S$ from $a$. For points $y$ that have not been reached by time $t^*$, we compare the difference in the time they are reached by $S$ and by $S'$. If $y \in Q_a$, it will be found at least time $L$ sooner by $S'$; if $y \notin Q_a$, it will be found exactly time $2 \delta'$ sooner by $S$. 

So, we have that
\begin{align*}
T(S', h) - T(S ,h) &\le \varepsilon (2\delta') - h(Q_a)(L) \\
& = 2 \delta' \left( \varepsilon - \frac{F(\delta') L}{2 \delta'} \right) \\
&\le 2 \delta' \left( \varepsilon - \frac{ML}{4} \right) < 0,
\end{align*}
by definition of $\varepsilon$. So $S'$ has a strictly smaller expected search time, contradicting the optimality of $S$. It follows that there must be some $\tau$ for which $S(\tau)=v$.

Since~(\ref{eq:leafy}) holds for all but one leaf node of $Q$, the argument above implies that all but one of the leaf nodes are visited by some $\tau'$. Since $S$ is optimal, it is clear that after visiting the penultimate leaf node, it must go directly to the final leaf node, so that $S$ is terminating. This completes the proof.
\end{proof}

We will also make use of the Search Density Lemma, whose proof can be found in~\cite{Alpern10}. Roughly speaking, it says that higher density searches should be carried out before lower density searches, to minimize expected search time. Many forms of this folk result exist in the literature; the following one is best for our purposes.

\begin{lemma}[Search Density Lemma] Suppose $h$ is a hiding distribution on a tree $Q$, and let $0 \le t_1 \le t_2 \le t_3 \le M$. Suppose $S([t_1,t_2])$ and $S([t_2,t_3])$ are probabilistically disjoint so that $h(S([t_1,t_3])) = h(S([t_1,t_2])) + h(S([t_2,t_3]))$ and let $S_1$ and $S_2$ be the restrictions of $S$ to $[t_1,t_2]$ and $[t_2, t_3]$, respectively.  Let $S'$ be the same as $S$ except that the order of $S_1$ and $S_2$ are swapped. Then
\[
\text{if } \rho(S_1) \ge \rho(S_2) \text{ then } T(S,h) \le T(S',h),
\]
with equality if and only if $\rho(S_1) = \rho(S_2)$.
\end{lemma}

\section{Balanced Distributions} \label{sec:GEBD}

In this section we give necessary and sufficient conditions on a hiding distribution that every DF search has the same expected search time, and we give a simple expression for this expected search time.

\begin{definition} \label{def:GEBD} We say a hiding distribution $h$ on a tree $Q$ is {\bf balanced} if at every branch node the search density of each branch is the same. 
\end{definition}
Note that if we define the function $g: x \mapsto \rho(Q_x)$ on $\{x \in Q: \lambda(Q_x) >0 \}$, then if $g$ is continuous it follows that $h$ is balanced. This is because otherwise $g$ would be discontinuous at some branch node.

A particularly important balanced distribution is the uniform distribution on $Q$. This is defined as the distribution (measure) which assigns to each arc $a$ probability proportional to its length $\lambda(a)$ and assigns to each subinterval of an arc $a$ measure proportional to its length (it is a multiple of Lebesgue measure on each arc). The uniform distribution is balanced because all branches have the same density as $Q_O$, which is $1/\mu$. 

Another important balanced distribution is the so called {\em Equal Branch Density} ({\em EBD}), distribution introduced by \cite{Gal79}. It is the unique distribution concentrated on the leaf nodes which is balanced. Gal showed that the EBD distribution is the hiding distribution $h$ that maximizes $\min_S T(S,h)$, and \cite{Alpern11} showed that the minimizing searches $S$ are exactly the DF searches. To illustrate the EBD distribution, consider the tree network depicted on the left of Figure~\ref{fig:network}. Nodes are labeled by letters, and lengths are shown alongside the arcs. Let the EBD distribution on this tree be denoted by $h$. For the tree to be balanced, the search density of the two branches at $O$ must be equal. Since the left branch has length $6$ and the right branch has length $8$, this means that we must have $h(A) = 6/14$ and $h(B)+h(C)=8/14$. In order for the two branches at $D$ to have equal search density, we must have $h(B)=2/5 \cdot 8/14$ and $h(C) = 3/5 \cdot 8/14$.

We now relate balanced distributions to DF searches.

\begin{lemma} \label{lem:DF} Let $h$ be a hiding distribution on a rooted tree. Then $h$ is balanced if and only if every DF search has the same expected search time.
\end{lemma}
\begin{proof}
	First suppose $h$ is balanced. Suppose $S_1$ and $S_2$ are two DF searches that differ only in that at some branch node, two branches are searched (in the same way) in a different order. By the last part of the Search Density Lemma, the expected search time of $S_1$ and $S_2$ is the same. Now suppose $S_1$ and $S_2$ are any two DF searches. By successively changing the order of search of pairs of branches at the same branch node, $S_1$ can be transformed into $S_2$ without changing the expected search time.
	
	Now suppose that not all DF searches have the same expected search time. Let $x$ be a branch node at maximal distance from $O$ such that there are two DF searches of $Q_x$ with different expected search times. (Such a node must exist, because $x=O$ satisfies this criterion.) Let $S_1$ and $S_2$ be two DF searches of $Q_x$ with different expected search times. Both searches must tour the branches of $Q_x$ is some order. Note that every DF search of a given branch of $Q_x$ must have the same expected search time, otherwise $d(O,x)$ would not be maximal. Let $S_2'$ be the search that searches the branches of $Q_x$ in the same order as $S_2$, but performs the same DF search of each given branch as $S_1$. Then $S_2$ and $S_2'$ have the same expected search time, and $S_2'$ has a different expected search time to $S_1$. By the Search Density Lemma, the branches of $Q_x$ cannot have the same search density, so $h$ cannot be balanced.
\end{proof}

We can now express the expected search time of a DF search against a balanced distribution in terms of a concise formula.
\begin{theorem} \label{thm:GEBD}
	If $h$ is a balanced distribution and $S$ is DF, then the expected search time $T(S,h)$ is given by
	\begin{align}
	T(S,h) = \mu - \int_{x \in Q} \lambda(Q_x)~dh(x) \label{eq:DF}.
	\end{align}
\end{theorem} 

\begin{proof}
	Let $S^{-1}$ be the time reverse of $S$, so that $S^{-1}(t) = S(2\mu - t)$ for $0 \le t \le 2\mu$. Let $s$ be the equiprobable choice of $S$ and $S^{-1}$, and denote the expected search time of $s$ by $T(s,h) = (T(S,h) + T(S^{-1},h))/2$. Note that $S^{-1}$ is also DF, so by Lemma~\ref{lem:DF}, $T(S^{-1},h) = T(S,h)$. Therefore,
	\begin{align}
	T(S,h)  &= T(s,h) \nonumber \\
	&= \int_{x \in Q} T(s,x)~dh(x) \nonumber \\
	& = \int_{x \in Q} \frac 1 2 (T(S,x) + T(S^{-1},x))~dh(x) \label{eq1}
	\end{align}
	Let $x \in Q$, and let $A$ and $B$ be the subnetworks searched by $S$ and $S^{-1}$ respectively up until reaching $x$ for the first time. Note that $A \cap B$ is equal to the path $P(O,x)$ from $O$ to $x$. Then before reaching $x$ for the first time, $S$ traverses all the arcs of $A \cap B$ exactly once in the forward direction and all other arcs of $A$ once in each direction; $S^{-1}$ traverses all arcs of $A \cap B$ once in the forward direction and all other arcs of $B$ once in each direction. Therefore, since $A \cup B = Q-Q_x$,
	\begin{align}
	T(S,x) + T(S^{-1},x) &= 2 \lambda(A \cup B) = 2\mu - 2\lambda(Q_x). \label{eq2}
	\end{align}
	Substituting (\ref{eq2}) into (\ref{eq1}) gives
	\begin{align*}
	T(S,h) &=  \int_{x \in Q} \mu - \lambda(Q_x)~dh(x) \\
	&= \mu -  \int_{x \in Q} \lambda(Q_x)~dh(x). 
	\end{align*}
\end{proof}

The search time $T(S,h)$ is maximized over $h$ by the EBD distribution, when the integral in~(\ref{eq:DF}) is equal to zero, since $Q_x$ has zero measure for leaf nodes $x$. In this case, the expected search time is simply equal to $\mu$. This is consistent with the expression for the worst-case expected search time for trees, as found in~\cite{Gal79}.

Equation~\ref{eq:DF} has a particularly nice form if the network $Q$ is a star: that is, a network consisting of $n$ arcs with one common node, $O$ (the root). This form can be found in \cite{Kella} (equation 3.13), but we include a derivation here based on~(\ref{eq:DF}) for completeness. For a hiding distribution $h$ on a tree $Q$ with root $O$, let $\bar{d}=\bar{d}_h(O) = \int_{x \in Q} d(O,x) dh(x)$ denote the average distance of points in $Q$ from $O$, with respect to $h$.

\begin{corollary} \label{cor:star}
Suppose a target is located on a star with $n$ arcs according to a balanced distribution $h$, and let $p_i$ be the probability the target is on the $i$th arc. Then any depth-first search $S$ has expected search time
\begin{align}
T(S,h) = \mu \left(1-\sum_{i=1}^n p_i^2 \right) +\bar{d}_h(O). \label{eq:star}
\end{align}
\end{corollary}
\begin{proof}
Let $Q_i$ denote arc $i$, and let $L_i=\lambda(Q_i)$ denote its length. By Equation~\ref{eq:DF}, we have
\begin{align*}
T(S,h) &= \mu - \sum_{i=1}^n \int_{x \in Q_i} \lambda(Q_x)~dh(x) \\
& = \mu - \sum_{i=1}^n \int_{x \in Q_i} (L_i-d(O,x))~dh(x)\\
& = \mu + \int_{x \in Q} d(O,x)~dh(x) - \sum_{i=1}^n L_i  \int_{x \in Q_i}~dh(x)\\
& = \mu + \bar{d} - \sum_{i=1}^n p_i L_i.
\end{align*}
Now, since each arc has equal search density, their densities must all be equal to the search density of the whole star, which is $1/\mu$. Hence $L_i = p_i \mu$ for each $i$, and Equation~(\ref{eq:star}) follows.
\end{proof}

We can also apply Theorem~\ref{thm:GEBD} to the special case of a uniform hiding distribution, $u$, given by $u(A) = \lambda(A)/\mu$, for measurable subsets $A$ of $Q$. Theorem 2 of \cite{Li-Huang} has already shown that for the uniform distribution on trees, a DF search is optimal (assuming no turns within arcs). Here we give a closed form expression for the expected search time of a DF search against the uniform distribution on a tree. Later, in Section~\ref{sec:monotone}, we will prove that even without the assumption of no turns within an arc, DF search is optimal for the uniform distribution.

\begin{corollary} \label{cor:uniform}
	Suppose a target is hidden on a tree according to the uniform distribution $u$. Then any DF search $S$ has expected search time
	\begin{align}
	T(S,u) = \mu - \bar{d}_u(O) \label{eq:unif}.
	\end{align}
\end{corollary}

\begin{proof}
	By Theorem~\ref{thm:GEBD}, it is sufficient to show that the integral in~(\ref{eq:DF}) is equal to $\bar{d}$. Noting that $\lambda(Q_x) = \int_{y \in Q_x} \mu~du(y)$, we can write
	\begin{align}
	\int_{x \in Q} \lambda (Q_x)~du(x) & = \int_{x \in Q} \int_{y \in Q_x} \mu~du(y)~du(x). \label{eq:int}
	\end{align}
	Now, for every point $y \in Q$, the set of points $x$ such that $y \in Q_x$ is exactly equal to $P(O,y)$. Therefore, swapping the order of integration on the right-hand side of~(\ref{eq:int}), we obtain
	\begin{align*}
	\int_{x \in Q} \lambda (Q_x)~du(x) & = \int_{y \in Q} \int_{x \in P(O,y)} \mu~du(x)~du(y) \\
	& = \int_{y \in Q} d(O,y)~du(y) \\
	& = \bar{d}. 
	\end{align*}
\end{proof}

Note that Corollary~\ref{cor:uniform} is not true in general, for non-uniform hiding distributions. For example, if $h$ is the EBD distribution, $T(S,h) = \mu$ but $\bar{d}_h(O)$ is not $0$.

\section{Monotone Hiding Distributions} \label{sec:monotone}

In this and the next section, we give conditions on the hiding distribution $h$ for some DF search to be optimal against it. 

\begin{definition} 
If some DF search is optimal against a hiding distribution $h$, we say $h$ is~{\bf simply searchable}. If the only optimal searches are DF, we say $h$ is {\bf strongly simply searchable}. 
\end{definition}
Note that if $h$ is simply searchable and balanced then all DF searches are optimal, by Lemma~\ref{lem:DF}. 

In this section we introduce a class of hiding distributions on trees we call {\em monotone} distributions, which are a subset of balanced distributions. We will show that DF searches are optimal against monotone distributions. 

\begin{definition}[\textbf{monotone}] \label{def:mono} 
We say the hiding distribution $h$ on a rooted tree $Q$ is monotone if for any $x \preceq y$, we have that $\rho(Q_x) \le \rho (Q_y)$.
\end{definition}

Clearly the uniform distribution and the EBD distribution are monotone. Also, it is easy to see that monotone distributions are leafy, since $\liminf_{x \rightarrow v} \rho( Q_x) \ge \rho(Q_O)= \rho(Q)= 1/\mu$ for all leaf nodes $v$. It follows from Theorem~\ref{thm:leafy} that any optimal search against a monotone hiding distribution is terminating.

Recall that $h$ has an atom at a point $x \in Q$ if $h(\{x\})>0$.

\begin{lemma} \label{lem:monotone}
Suppose $h$ is a monotone hiding distribution. Then
\begin{enumerate}
\item[(i)] $h$ has no atoms except possibly at leaf nodes;
\item[(ii)] $\rho(Q_x)$ is continuous in $x$ on the set containing all points of $Q$ except leaf nodes, where it is not defined;
\item[(iii)] $h$ is a balanced distribution.
\end{enumerate}
\end{lemma}

\begin{proof}
For (i), suppose there is an atom of measure, say, $\varepsilon$ at some point $x$ that is not a leaf node. Suppose $x$ has degree $n \ge 2$, and let $y_1,\ldots,y_{n-1}$ be points on the $n-1$ arcs above $x$ satisying $d(x,y_j) < \varepsilon/(n\rho(Q_x)), j=1,\ldots,n-1$. Then 
\[
\rho(\cup_j Q_{y_j}) =\frac{h(\cup_j Q_{y_j})}{\lambda(\cup_j Q_{y_j})} < \frac{h(Q_x) - \varepsilon}{\lambda(Q_x) - \varepsilon/\rho(Q_x)} = \rho(Q_x)
\]
Since $\rho(\cup_j Q_{y_j})$ is a weighted average of each $\rho(Q_j)$, there must be some $j$ for which $\rho(Q_j)$ is strictly less than $\rho(Q_x)$, contradicting monotonicity. Roughly, this means we would have $\rho(Q_x-\{x\}) < \rho(Q_x)$.

This establishes (i); (ii) is a consequence of this; (iii) follows from (ii) and the remark following Definition~\ref{def:GEBD}. 
\end{proof}


We can give an equivalent characterization of monotone distributions in the case that $h$ can be described by a probability density. In particular, consider any path $P$ from the root $O$ to some leaf node. Suppose that $h$ has a pdf $f$, so that $\int_{0}^t f(x)~dx$ is the probability that the target is on $P$ at most distance $t$ from $O$. Then it can be shown that $h$ is monotone if and only if for every such path $P$ with pdf $f$, 
\[
f(t) \le \rho(Q_{x(t)}),
\]
for all $t$, where $x(t)$ is the unique point on $P$ at distance $t$ from $O$. This can be proved rigorously, but it is also intuitively clear from considering the graph in Figure~\ref{fig:monotone}. This corresponds to a path $P$ with $h(P)=1/2$. The solid red line is the cdf $F$ of a monotone distribution on $P$, whose slope is the pdf $f$. The search density of a subtree $Q_{x(t)}$ is given by the slope of the dotted line segment that goes from the point $(t,F(t))$ to $(1,F(1))$. For $h$ to be monotone, the slope of these lines must be non-decreasing in $t$, or equivalently, the slope $f(t)$ of the red line must be no greater than that of the dotted lines.

\begin{figure}[h!]
  \centering
    \includegraphics[width=0.4\textwidth]{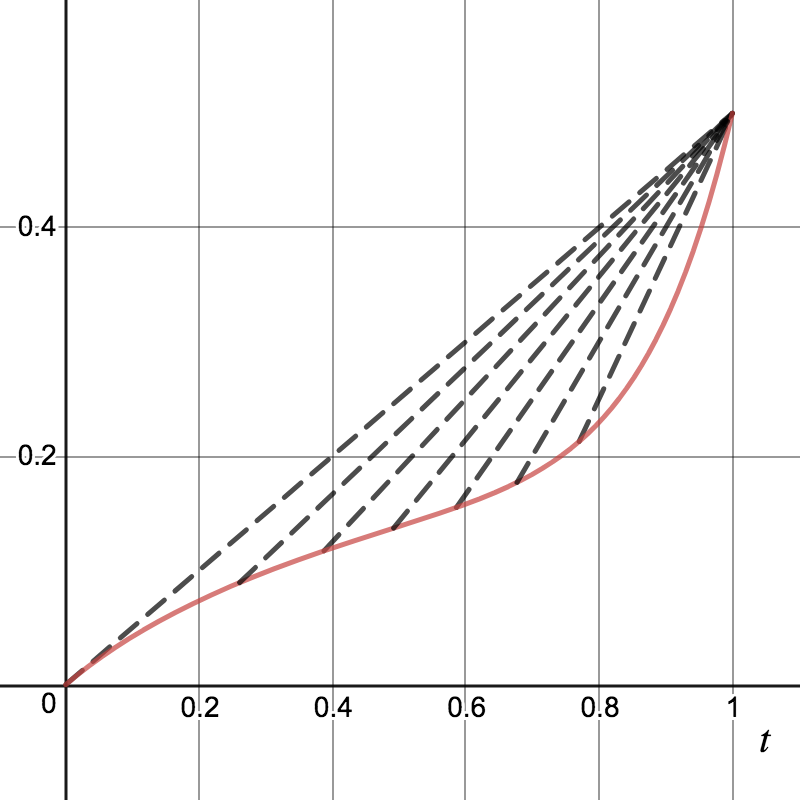}
      \caption{The cdf of a monotone distribution on a path from $O$ to a leaf node (in solid red).}       \label{fig:monotone}
\end{figure}

It is also worth pointing out that monotonicty implies that $F$ has finite derivative at $0$. This is clear from Figure~\ref{fig:monotone}, and can be proved rigorously. Therefore, as in Theorem 12 of \cite{Beck2}, no optimal search can start with infinite oscillations. In fact, we do not need to use this observation in what follows.

We will show that DF searches are optimal against monotone distributions, but to do so we need a lemma about monotone distributions. 

\begin{lemma} \label{lem:subtree}
Suppose $h$ is a monotone distribution on a tree $Q$. If $A$ is a subtree of $Q$ containing $O$, then $\rho(A) \le \rho(Q)$.
\end{lemma}
\begin{proof}
The complement $\bar{A}$ of $A$ is a disjoint union of subtrees $Q_x$ of $Q$. Since $h$ is monotone, the search density of these subtrees is at least the search density of $Q$, so $\rho(\bar{A}) \ge \rho(Q)$. Since $\rho(Q)$ is a weighted average of $\rho(A)$ and $\rho(\bar{A})$, it follows that $\rho(A) \le \rho(Q)$.
\end{proof}



We can now prove that against monotone distributions, it is optimal to use a DF search. 
\begin{theorem} \label{thm:monotone}
Let $h$ be a monotone hiding distribution on a tree $Q$. Then $h$ is strongly simply searchable.
\end{theorem}
\begin{proof}
Suppose $h$ is not strongly simply searchable, and let $S$ be an optimal search that is not DF. Since $h$ is monotone, it is leafy, so $S$ must be terminating, by Theorem~\ref{thm:leafy}. By Lemma~\ref{lem:DFdecom}, the non-DF set $N=N_S$ of $S$ is a closed subtree of $Q$ containing $O$ (since $S$ is not DF) and containing none of the leaf nodes of $Q$. Let $x$ be a leaf node of $N$. Since $Q_x-\{x\}$ is a subset of the DF set $D=D_S$, it must be the case that $S$ performs DF searches of all the branches of $Q_x$. Therefore, there must be times $t_1, t_2$ with $t_1 < t_2$ such that $S$ is disjoint from $Q_x$ in the interval $(t_1,t_2)$, and then it performs a DF search of a branch $B$ of $Q_x$ starting at time $t_2$, where $B$ may be equal to $Q_x$. In any case, we must have $\rho(B)=\rho(Q_x)$, since $S$ is balanced. 


Let $S_1$ be the search $S$ restricted to $[t_1,t_2]$, and let $S_2$ be the DF search of $B$ starting at time $t_2$. Let $A$ denote the set covered by $S_1$ and let $y$ be the lowest point of $A$. Then since $A$ is a subtree of $Q_y$ containing $y$, we have $\rho(A) \le \rho(Q_y)$, applying Lemma~\ref{lem:subtree} to $Q_y$ and $A$. It follows that the search density of $S_1$ satisfies
\[
\rho(S_1) < \frac{1}{2} \rho(Q_y),
\]
where strictness follows from the fact that part of $S_1$ retraces arcs. On the other hand,
\[
\rho(S_2)  = \frac{1}{2} \rho(B).
\]
The monotonicity of $h$ ensures that $\rho(B)=\rho(Q_x) \ge \rho(Q_y)$, hence $\rho(S_2) > \rho(S_1)$.

By the Search Density Lemma, transposing $S_2$ and $S_1$ at time $t_1$ reduces the expected search time. But this results in a new search $S'$ with a strictly smaller expected search time than $S$, contradicting the optimality of $S$. So $h$ is strongly simply searchable. 
\end{proof}

Theorem~\ref{thm:monotone} implies that if $h$ is a monotone hiding distribution, then $V(h)$ is given by Equation~(\ref{eq:DF}). This follows from Theorem~\ref{thm:GEBD} and Lemma~\ref{lem:monotone}, part (iii).

We can now prove a more general version of Theorem 2 from \cite{Li-Huang}, and give a simple method to choose the starting point for the search that has least optimal expected search time. 
\begin{corollary} \label{cor:start}
Suppose a target is hidden on a tree according to the uniform distribution $u$. Then
\begin{enumerate}
\item[(i)] $u$ is strongly simply searchable and
\[
V(u) = \mu - \bar{d}(O), \text{ and}
\]
\item[(ii)] the choice of root $O$ that minimizes the expected search time of any DF search is the leaf node $x$ that maximizes $\bar{d}(x)$.
\end{enumerate}
\end{corollary}
\begin{proof}
Part~(i) follows from Theorem~\ref{thm:monotone}, Corollary~\ref{cor:uniform} and the fact that the uniform distribution is monotone.

Part~(ii) follows from the observation that $\bar{d}(x)$ is maximized when $x$ is a leaf node. Indeed, suppose not, and that $\bar{d}(x)$ is maximized at some point $x$ which is not a leaf node. Let $R$ be a connected component of $Q-\{x\}$ with $\lambda(R) \le \lambda (Q-R)$, and let $y$ be a point in $R$ on the same arc as $x$ with $d(x,y)=\varepsilon$. Then it is easy to see that 
\[
\bar{d}(y)-\bar{d}(x) =  \frac{\lambda(Q-R)}{\mu}(\varepsilon) + \frac{\lambda(R) - \varepsilon}{\mu}(-\varepsilon) \ge \frac{\varepsilon^2}{\mu} >0,
\]
contradicting the maximality of $\bar{d}(x)$.
\end{proof}

It should be emphasized that~\cite{Li-Huang} already showed directly that the optimal choice of root for the uniform hiding distribution is some leaf node, but they did not show how to determine {\em which} leaf node is optimal.

Note that evaluating the mean distance $\bar{d}(x)$ from $x$ to other points in the network is equivalent to evaluating the average distance from $x$ to the midpoints of the arcs, weighted by the lengths of the arcs. If the network has unit length arcs, comparing these averages is equivalent to comparing the mean distance from $x$ to all other nodes, since both measures induce the same ordering on the nodes. If the lengths of the arcs are all rational, then by adding nodes of degree 2 to the network, it can be transformed into a network with unit length arcs. In this case, the problem of finding the leaf node $x$ that maximizes $\bar{d}(x)$ is equivalent to finding the leaf node of a graph whose mean distance to all other nodes is minimized. This, in turn is equivalent to finding the node $x$ of minimal {\em closeness centrality}, which is defined as the reciprocal of the mean distance from $x$ to all other nodes. Closeness centrality was introduced by \cite{Bavelas} and is used widely in social network analysis.

\section{Optimal Depot Location in the Delivery Man Problem}
\label{sec:depot}

We define the {\em equiprobable distribution} $e$ as the atomic distribution that places equal weight on each node (including the root node). The problem of finding the optimal search on a general network against the equiprobable distribution is known as the {\em Delivery Man Problem} or {\em Traveling Repairman Problem}.  Although the distribution $e$ is not monotone, \cite{Minieka} already showed directly that for a tree with equal arc lengths any DF search is optimal. We can use Corollary~\ref{cor:start} to improve upon this.
\begin{theorem}
For the Delivery Man Problem on a tree with $n$ unit length arcs,
\begin{enumerate}
\item[(i)] the optimal expected search time is 
\[
V(e)=n-1 - \bar{d}_e(O),
\]
where $\bar{d}_e(O)$ is the mean distance from $O$ to all {\bf nodes} of the network (including~$O$);
\item[(ii)] the optimal choice of depot for the Delivery Man Problem is the leaf node $x$ of minimum closeness centrality.
\end{enumerate}
\end{theorem} 
\begin{proof}
For part (i), consider the network $Q'$ obtained by adding an extra unit length arc, one of whose endpoints is $O$, and the other is a new degree $1$ node $O'$. The average distance from $O'$ to points in $Q'$ is $\bar{d}_e(O) + 1/2$, so by Corollary~\ref{cor:start}, the expected search time of a DF search of $Q'$ against the uniform distribution $u$ is $V(u) = \lambda(Q') - (\bar{d}_e(O)+1/2)$. Using $\lambda(Q')=n$ and $V(e)=V(u)-1/2$, the result follows.

Part (ii) follows from the fact that $\bar{d}_e(x)$ is maximized at a leaf node.
\end{proof}

\cite{Minieka} has claimed that ``For a tree with equal edge weights, the best depot is the endpoint of any longest path.''  While in many cases this will agree with our choice of best location, we present in Figure~\ref{fig:start} a network with unit arc lengths for which the leaf node of minimal closeness centrality is a strictly better place to start than any end of a path of maximum length. First we show that $C$ is the node of minimal closeness centrality and then we show directly that it is a better depot location (starting point) than the node $A$ at the end of the maximum length path. 

There are two contenders (up to symmetry) for nodes of minimum closeness centrality: nodes $A$ and $C$. The distances from $A$ to the other nodes are written in green on the top left of each node, and the distances from $C$ are written in red on the bottom right. The sum of the distances from $A$ is $46$ and from $C$ is $47$. So $C$ has the smallest closeness centrality, and is therefore the best choice of depot for the Delivery Man Problem. Indeed, when following a DF search from $C$, the sum of the times to reach the other $12$ nodes is $1+3+5+7+9+13+17+25+27+29+33+37=206$. The corresponding sum when starting at $A$ is $1+3+7+11+13+15+17+23+25+27+31+35=208$, which confirms directly that it is best to start from $C$.

\begin{figure}[h]
	\centering
	\includegraphics[width=0.5\textwidth]{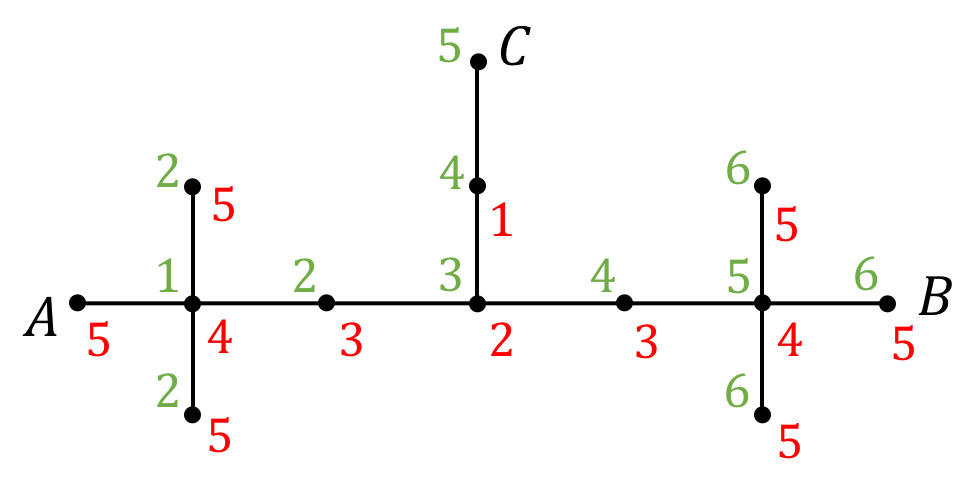}
	\caption{A network with distances labeled from node $A$ in green and from $C$ in red.}       \label{fig:start}
\end{figure}

For some hiding distributions it is best to adopt the DF search which starts and ends at diametrical points of the tree, such as $A$ and $B$ in Figure~\ref{fig:start}. \cite{DG08} considered the search game where the Hider picks any point of the tree and the Searcher can start at any point. They showed that the optimal strategy for the Hider was the distribution $h^{\ast }$, which is the Equal Branch Density distribution when taking the root as the {\em center} of the tree (the point minimizing the maximum distance to other points, or the midpoint of diametrical points). This results in a hiding distribution that places probability $6/36$ at $C$ and probability $5/36$ at the six other leaf nodes. For the Searcher, the optimal mixed strategy is to take a Chinese Postman Path (on a tree this starts and ends at diametrical points) and traverse it equiprobably in either direction. Thus starting at $A$ or $B$ is optimal against the distribution $h^{\ast }$. Clearly this is also the solution to the game where the Hider must choose a node and the Searcher must start at a node. 

Thus for any tree $Q$, the Dagan-Gal solution gives a hiding distribution $h^*(Q)$ and a start point $A(Q)$ such that all optimal searches are DF but not all DF searches are optimal. For example, for the tree of Figure~\ref{fig:start} with hiding distribution $h^*$, a DF search starting at $A$ is optimal if and only if the last point reached is diametrical to $A$.

It is worth noting that on a star, the node of minimum closeness centrality is always located at the end of the longest arc. \cite{Kella} showed that the optimal starting position on a star is at the end of the longest arc for a class of hiding distributions that includes the uniform distribution.

\section{Which Distributions Are Simply Searchable on a Star?} \label{sec:star}

In this section we restrict our attention to stars. A star is a tree with exactly one node of degree greater than $1$, and we always assume this node is the root $O$ in this section. We consider the question of what are necessary and sufficient conditions on the hiding distribution for it to be simply searchable. To that end, we define a class of hiding distributions on a star.

\begin{definition}[\textbf{forward biased}]
Let $h$ be a hiding distribution on a star with arcs $j=1,\ldots ,n$ of lengths $\lambda _{j}$ with $\sum_j \lambda_j = \mu$. Let $F_{j}(x) $ be the probability that the target is on arc $j$ at distance from the root less
than or equal to $x$ and let $h_j = F_j(\lambda_j)$ be the probability the target is located on arc $j$. We say that $h$ is forward biased if for all $j$ we have
\begin{align}
F_{j}(x) \leq H_{j}(x) \equiv \frac{x+(h_{j} \mu -\lambda _{j}) ^{+}}{x+\mu -\lambda _{j}},\text{ for all } x\leq \lambda _{j},\text{ where }y^{+}=\max \{ y,0\} . \label{eq:fb}
\end{align}
If condition~(\ref{eq:fb}) is strict for all $j$ and $x < \lambda_j$, then we say $h$ is strictly forward biased.
\end{definition}

A condition of the type $F_{j}(x) \leq H_{j}(x) $ puts an upper bound on how likely the target is close to the root on an arc.
So it is more likely to be near to the forward (leaf node) part of the arc. This is the reason for the name.

We show in Subsection~\ref{sec:balanced-star} that a balanced hiding distribution on a star is simply searchable if and only if it is forward biased. In Subsection~\ref{sec:interval}, we remove the assumption of balanced in the case that the network is a line segment, and show a hiding distribution is simply searchable if and only if it is forward biased.

\cite{Kella} also considered the problem of when DF search is optimal on a star, giving a sufficient condition on the hiding distribution for it to be simply searchable. In Section~\ref{sec:kella}, we consider Kella's condition, and show that it is stronger than ours.

We first show that the distributions we consider in this section are leafy.

\begin{lemma} \label{lem:leafystar}
Let $h$ be a forward biased hiding distribution on a star with $k$ arcs. If $h$ is balanced or $k=2$, then $h$ is leafy.
\end{lemma}
\begin{proof}
In the first case, that $h$ is balanced, every arc of the star must have the same search density as the whole star, which is $1/\mu$. Since the search density of arc $j$ is $h_j/\lambda_j$, this implies that $(h_{j} \mu -\lambda _{j}) ^{+} = 0$ for all $i$, and condition~(\ref{eq:fb}) reduces to
\begin{align}
F_{j}(x) \leq H_{j} (x) \equiv \frac{x}{x+\mu -\lambda _{j}}, \text{ for all } x\leq \lambda _{j}.  \label{eq:DF-cond}
\end{align}
It follows that the search density of the region within distance $\varepsilon$ of the leaf node of an arc $j$ is
\[
\frac{h_j-F_j(\lambda_j - \varepsilon)}{\varepsilon} \ge \frac{h_j (\mu -\varepsilon) - \lambda_j +\varepsilon}{(\mu - \varepsilon)\varepsilon} = \frac{1-h_j}{\mu-\varepsilon} \rightarrow \frac{1-h_j}{\mu}, \text{ as } \varepsilon \rightarrow 0.
\]
So $h$ is leafy.

In the second case, we only need to show that the limiting search density of the tip of {\em one} of the arcs is positive. But one of the two arcs $j$ must have search density at most $1/\mu$, so that $(h_{j} \mu -\lambda _{j}) ^{+} = 0$, and the same argument as above holds.
\end{proof}

It follows from Lemma~\ref{lem:leafystar} and Theorem~\ref{thm:leafy} that any optimal search for a target hidden according to a forward biased distribution on a star is terminating if the distribution is balanced or has two arcs. 

\subsection{Balanced stars} \label{sec:balanced-star}

In this subsection we assume that the hiding distribution is balanced, so that some DF is optimal if and only if all DF searches are optimal. 
We show that a balanced hiding distribution on a star is simply searchable if and only if it is forward biased.
\begin{theorem} \label{thm:star}
Suppose a target is located on a star according to a balanced hiding distribution~$h$. Then~$h$ is simply searchable if and only if it is forward biased. Moreover, $h$ is strictly simply searchable if and only if it is strictly forward biased.
\end{theorem}

\begin{proof}
We first show that if $h$ is not forward biased then it is not simply searchable. Suppose that condition~(\ref{eq:DF-cond}) does not hold for some point $P$ on arc $j$ at distance $x < \lambda_j$ from the $O$. Then let $S^*$ be the non-DF search $S_1,S_2,S_3$, where $S_1$ goes directly from $O$ to $P$, then $S_2$ returns to $O$ and tours the remaining arcs of the star before returning to $P$, and $S_3$ tours $Q_P$. Then the search density of $S_2$ is 
\begin{align}
\rho(S_2) = \frac{1 - \lambda_j/\mu}{2(\mu- \lambda_j + x)}. \label{eq:S2-density}
\end{align}
The search density of $S_3$ is
\begin{align}
\rho(S_3) & =\frac{\lambda_j/\mu - F_j(x)}{2(\lambda_j - x)}. \label{eq:S3-density}
\end{align}
Therefore, the difference between the search density of the two searches is
\begin{align}
\rho(S_2) - \rho(S_3) = \frac{1}{2(\lambda_j - x)} \left( F_j(x) - \frac{x}{\mu - \lambda_j + x} \right).  \label{eq:density-diff}
\end{align}
This difference is positive, since condition~(\ref{eq:DF-cond}) does not hold. Therefore, the DF search $S'$ obtained from $S$ by swapping the order of $S_2$ and $S_3$ has a greater expected search time than $S$, by the Search Density Lemma, so $h$ cannot be simply searchable. 

If $h$ is not strictly forward biased but it is forward biased, then condition~(\ref{eq:fb}) holds with equality for some arc $j$ and some distance $x$. In this case,~(\ref{eq:density-diff}) holds with equality, and $S'$ must be optimal. Hence, $h$ is not strictly simply searchable.

Now suppose $h$ is forward biased, and we will show it is simply searchable. Let $S$ be an optimal search that is not DF. By Lemma~\ref{lem:leafystar}, it must be the case that $h$ is leafy, so $S$ must be terminating, by Theorem~\ref{thm:leafy}. By Lemma~\ref{lem:DFdecom}, the non-DF set $N=N_S$ of $S$ is a closed subtree of $Q$ containing $O$ (since $S$ is not DF) and containing none of the leaf nodes of $Q$. Let $P$ be the leaf node of $N$ with the largest expected search time $t_1=T(S,P)$. Again, we express $S$ as a succession of three searches, $S_1,S_2,S_3$. The first, $S_1$ follows $S$ from time $t=0$ until time $t_1$. The second, $S_2$ starts at time $t_1$ and ends when $S$ next reaches $P$ at time $t_2$. The third, $S_3$, tours $Q_P$, starting at time $t_2$. Note that $S_2$ must go from $P$ to $O$ and then perform a DF search of some set $A$ of arcs of the star.

The search density of $S_2$ satisfies,
\begin{align}
\rho(S_2) \le \frac{\sum_{i \in A} \lambda_i/\mu}{2(\sum_{i \in A} \lambda_i + x)} \le \frac{1 - \lambda_j/\mu}{2(\mu- \lambda_j + x)},\label{eq:S2-density2}
\end{align}
where both the inequalities in~(\ref{eq:S2-density2}) holds with equality if and only if $S$ is equal to the search $S^*$ from the first paragraph of the proof. Also, $\rho(S_3)$ is given by the Equation~(\ref{eq:S3-density}), so the right-hand side of Equation~(\ref{eq:density-diff}) is an upper bound for $\rho(S_2) - \rho(S_3)$. This upper bound holds with equality if and only if $S=S^*$, in which case the search $S'$ obtained by swapping the order of $S_2$ and $S_3$ is a DF search with the same expected search time, and is, therefore, optimal. Otherwise the bound is strict, and $S'$ has a strictly smaller expected search time, contradicting the optimality of $S$. So $h$ is simply searchable.

If $h$ is strictly forward biased, then the right-hand side of Equation~(\ref{eq:density-diff}) must be a strict upper bound for $\rho(S_2) - \rho(S_3)$, so that $S^*$ cannot be optimal. Hence, the only optimal searches are DF and $h$ is strictly simply searchable.
\end{proof}

If follows from Theorem~\ref{thm:star} that for a forward biased hiding distribution $h$ on a star, $V(h)$ is given by~(\ref{eq:star}).



\subsection{Two-arc stars (intervals)} \label{sec:interval}

We now remove the assumption that $h$ is balanced, and consider the same question as in the previous subsection: what conditions are necessary and sufficient for $h$ to be simply searchable on a star? Notice that for arcs $i$ whose search density $h_i/\lambda_i$ is lower than the average search density $1/\mu$ of the star, condition~(\ref{eq:fb}) reduces to~(\ref{eq:DF-cond}). For arcs with higher than average search density, the term $h_i \mu - \lambda_i$ is included in the numerator of~(\ref{eq:fb}), so that the bound is
\begin{align}
F_{j}(x) \leq H(x) \equiv \frac{x+h_{j} \mu -\lambda _{j} }{x+\mu -\lambda _{j}},\text{ for all } x\leq \lambda _{j}. \label{eq:DF-cond2}
\end{align}

We will restrict our attention here to two-arc stars, which can be represented as an interval ${Q=[ -\lambda_2,+\lambda_1] }$ containing the root $O=0$. We refer to the subinterval $[0,\lambda_1]$ as the {\em right arc}, and the subinterval $[-\lambda_2,0]$ as the {\em left arc}. Throughout this section we assume that the search density $ h_1/\lambda_1$ of the right arc is at least the average search density $1/\mu$, so that $h$ is simply searchable if and only if the DF search $S^+$ that starts with the right arc is optimal. This means that for the right arc, condition~(\ref{eq:fb}) takes the form~(\ref{eq:DF-cond2}) and for the left arc, it takes the form~(\ref{eq:DF-cond}). Note that the right-hand side of~(\ref{eq:DF-cond2}) is bounded above by $h_1$, but the right-hand side of~(\ref{eq:DF-cond}) may be strictly greater than $h_2$ for some values of $x$ (in particular, $x=\lambda_2$, where it is equal to $\lambda_2/\mu \ge h_2$). Thus the cdf $F_2$ on the left must satisfy the stricter condition $F_2(x) \le \min \{H(x), h_2 \}$. 

The two forms of the constraint~(\ref{eq:fb}) are illustrated in Figure~\ref{fig:interval} for a star with two arcs both of length 1 and the weight $h_1$ on the right equal to $2/3$. The solid lines represent the bound $H(x)$ for $0 \le x \le \lambda_1$ and $\min \{H(-x), 1/3 \}$ for $-\lambda_2 \le x \le 0$ . The dashed lines show the cdfs $F_1$ and $F_2$ for the uniform distribution with weight $1/3$ on the left and weight $2/3$ on the right.

\begin{figure}[h]
  \centering
    \includegraphics[width=0.4\textwidth]{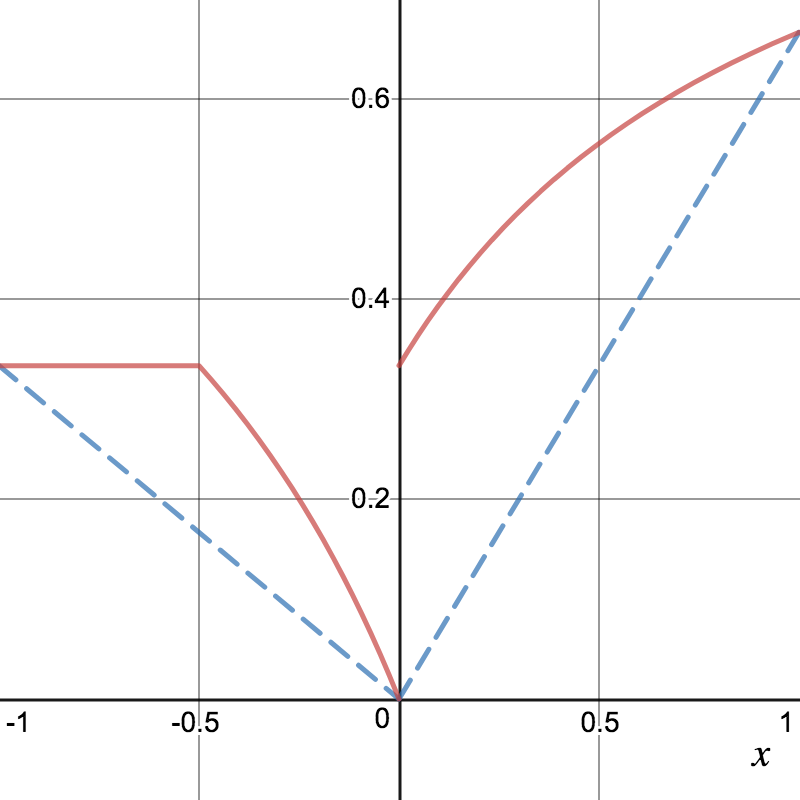}
      \caption{Uniform distribution (dashed) on the interval for $h_1=2/3$ and the bounds $H(x)$ (right) and $\min \{H(-x),1/3 \}$ (left) in solid.}       \label{fig:interval}
\end{figure}

We show that for two-arc stars, $h$ is simply searchable if and only if it is forward biased. To prove this we first show that if we restrict our searches to having at most one turning point within an arc, then $h$ is simply searchable if and only if it is forward biased.

We first define two single-turn searches $S^{\ast }(x)=[0,x,-\lambda_{2},+\lambda_{1}] $ and $\hat{S}( y) =[0,-y,+\lambda_{1},-\lambda_{2}]$, $0<x<\lambda_1,0<y<\lambda_2$, where the points listed in the square bracket refer to the turning points. (The $0$ is there to indicate the searches start at $0$.) We will compare these searches to the best DF search, $S^{+}=[0,+\lambda_1,-\lambda_2]$.

\begin{lemma} \label{lem:FB}
For any fixed $x$ and $y$ with $0 < x < \lambda_1, 0< y < \lambda_2$, 
\begin{enumerate}
\item[(i)] the expected search time of $S^{\ast }(x)$ is smaller than that of $S^{+}$ if and only if condition~(\ref{eq:fb}) fails for this $x$ and $j=1$, and
\item[(ii)] the expected search time of $\hat{S}(y)$ is smaller than that of $S^{+}$ if and only if condition~(\ref{eq:fb}) fails for this $y$ and $j=2$.
\end{enumerate}
\end{lemma}

\begin{proof}
For part (i), we observe that for fixed $x$, after the search $S^*(x)$ goes from $0$ to $x$, it continues with $S_{1}=[ x,-\lambda_{2},x] $ and then $S_{2}=[ x,\lambda_{1},x] $. By the Search Density Lemma, we know that $T(S^{\ast }(x),h) <T( S^{+},h) $ if and only if $\rho(
S_{1}) > \rho( S_{2})$. 
It is easily seen that
\[
\rho ( S_{1})  = \frac{1-h_1}{2(x+\lambda_{2})} \text{ and } \rho ( S_{2})  =\frac{h_1-F_1( x) }{2(\lambda_{1}-x)}.
\] 
If follows that
\[
\rho ( S_{1}) >\rho ( S_{2}) \text{ if and only if } F_1( x) >  \frac{x+h_1 \mu -\lambda_{1} }{x+\lambda_{2}} \equiv H(x).
\]
Hence, $\rho ( S_{1}) >\rho ( S_{2}) $ if and only if~(\ref{eq:DF-cond2}), and hence~(\ref{eq:fb}) fails for this $x$ and $j=1$.

For part (ii), define the sets $A=[ 0,\lambda_{1}] ,~B_{1}=[ -y,0] ,~B_{2}=[ -\lambda_{2},-y] $ and define the time
difference $\Delta ( z) =T( \hat{S}( y) ,z)
-T( S^{+},z)$, for $z \in [-\lambda_2, \lambda_1]$. For $z$ in $A,$ $B_{1},$ $B_{2}$ we have
\begin{align*}
\text{ If }z &\in A, \text{ then } \Delta =2y \text{ (comes later, after going to $y$ and back);} \\
\text{if } z &\in B_{1}, \text{ then } \Delta =-2\lambda_{1} \text{ (comes earlier, before going to 1 and back);} \\
\text{ if }z &\in B_{2}, \text{ then } \Delta =2y \text{ (comes later, having gone to $y$ and back an extra time).}
\end{align*}
Since the measures of the three sets are given by $h( A)=h_1,~h( B_{1}) =F_2( y) ,~h( B_{2}) =(1-h_1) -F_2( y) $, it follows that the expected value of $\Delta $ is given by 
\[
T( \hat{S}( y) ,h) -T( S^{+},h) =2y(h_1+( 1-h_1) -F_2( y) ) -2 \lambda_{1}F_2( y) .
\]
So $T(\hat{S}( y),h) $ is smaller than $T(S^{+},h)$ if and only this expression is negative, which, on solving for $F_2( y) $, gives $F_2( y) >y/(\lambda_{1}+y) \equiv H(y)$. Equivalently~$(\ref{eq:DF-cond})$, and hence~(\ref{eq:fb}) fails for this $y$ and $j=2$.
\end{proof}


\begin{theorem} \label{thm:FB}
A hiding distribution $h$ on a two-arc star is simply searchable if and only if it is forward biased. Moreover, $h$ is strictly simply searchable if and only if it is strictly forward biased.
\end{theorem}

\begin{proof}
First suppose that $h$ is not forward biased. Then condition~(\ref{eq:fb}) fails for some $j$ and some $x < \lambda_j$. In this case, by Lemma~\ref{lem:FB}, either the search $S^{\ast }( x) $ or the search $\hat{S}(y)$ has a smaller expected search time than that of $S^{+}$. In either case $S^{+}$ is not optimal, and therefore, no DF search is optimal and $h$ is not simply searchable. If $h$ is not strictly forward biased, but it is forward biased, then $S^{\ast}(x)$ or $\hat{S}(y)$ have the same expected search time as $S^+$ for some $x$ or $y$, so $h$ is not strictly simply searchable.

On the other hand, suppose $h$ is forward biased. Let $S$ be an optimal search, and suppose $S$ is not DF. By Lemma~\ref{lem:leafystar}, the hiding distribution $h$ is leafy, so $S$ must be terminating, by Theorem~\ref{thm:leafy}. The non-DF set $N=N_S$ is some interval $[-y,x]$ with $-1< -y \le 0 \le x <1$. If $y=0$ then $S=S^*(x)$ and if $x=0$ then $S=\hat{S}(y)$. In either case, by Lemma~\ref{lem:FB}, the DF search $S^+$ also optimal, and $h$ is simply searchable. So assume that $-y < 0 < x$ and we will derive a contradiction. Note that we must have $F_1(x),F_2(y) >0$, otherwise $\hat{S}(y)$ or $S^*(x)$, would have a strictly smaller expected search time that $S$.

First suppose that $T(S,x) > T(S,-y)$. Then at time $T(S,x)$, the search $S$ must follow $S_1 = [x, -\lambda_2, x]$ followed by $S_2 = [x, \lambda_1, x]$. (We may as well assume that $S$ returns to $x$ after reaching $\lambda_1$.) The search density of $S_1$ is
\[
\rho(S_1)  = \frac{1-h_1 - F_2(y)}{2(x+\lambda_2)} < \frac{1-h_1}{2(x+\lambda_2)} \text{ and } \rho(S_2) = \frac{h_1 - F_1(x)}{2(\lambda_1-x)}.
\]
Since $h$ is forward biased, it follows that $\rho(S_2) > \rho(S_1)$, similarly to the proof of Lemma~\ref{lem:FB}. Hence, $S_1$ and $S_2$ can be swapped to obtain a search with a strictly smaller expected search time, contradicting the optimality of $S$.

Now suppose that $T(S,-y) > T(S,x)$. Note that $S$ must go directly from $x$ to $-y$ between times $T(S,x)$ and $T(S,-y)$. Let $h'$ be the marginal hiding distribution on the interval after time $T(S,x)$. Regarding $x$ as the new root, let $\rho_1$ and $\rho_2$ be the densities of the new right and left arcs, $[x, \lambda_1]$ and $[-\lambda_2, x]$, with respect to the new hiding distribution $h'$. Then $h'([x,\lambda_1]) = h_1 - F_1(x)$ and $h'([-\lambda_2, x]) < 1-h_1$, since $F_1(x)>0$. Rearranging condition~(\ref{eq:DF-cond2}) for $j=1$, we get
\[
\frac{h_1 - F_1(x)}{\lambda_1 - x} \ge \frac{1-h_1}{\lambda_2 + x },
\]
and it follows that $\rho_1 > \rho_2$.

We show that $h'$ is forward biased on the interval with root $x$. Let $p \le 1-F_1(x)$ be the probability that the target has not been found before $S$ reaches $x$, and let $F'_1$ and $F'_2$ be the cdfs on the new right and left arcs, with respect to $h'$. Since the right arc has higher search density than the left we need to establish condition~(\ref{eq:DF-cond2}) for the right arc, which says
\[
F'_1(z) \equiv \frac{F_1(z+x)-F_1(x)}{p} \le \frac{z + \mu (h_1 - F_1(x))/p - (\lambda_1 -x)}{z + (\lambda_2 + x)},
\]
where $z$ is the distance from $x$ to a point on the (new) right arc. It is easy to show that this is equivalent to the condition
\[
F_1(y) \le H_1(y) + \frac{ (1-F_1(x)-p)(\lambda_1 - y)}{y+\lambda_2},
\]
where $y=x+z$. Since we know that $F_1(y) \le H_1(y)$, it is sufficient to show that the sum of the remaining terms on the right-hand side of the expression above is non-negative. This is equivalent to the condition $p \le 1-F_1(x)$, which we have already noted. 

For the left arc, we need to show that condition~(\ref{eq:fb}) holds for $F'_2$. This is trivially true for any point on the left arc at distance $z \le x$ from the (new) root, since then $F'_2(z)=0$. So consider a point at distance $z > x$ from the root. Then
\[
F'_2(z) \le F_2(z-x) \le H_2(z-x) = \frac{z-x}{z-x+ \lambda_1} < \frac{z}{z + (\lambda_1 - x)}.
\]
This establishes condition~(\ref{eq:fb}) for $h'$, and furthermore the condition holds strictly on the left arc. It follows from Lemma~\ref{lem:FB} that when $S$ reaches $x$, it would be better (smaller expected search time) to continue to $+\lambda_1$, then to $-\lambda_2$. 

We leave it to the reader to check that if $h$ is strictly forward biased then it is strictly simply searchable.
\end{proof}

\subsection{Kella's condition for simply searchable stars} 
\label{sec:kella}

\cite{Kella} also considers the question of which hiding distributions are simply searchable on a star. In Theorem 3.1 he gives a sufficient condition for simple searchability. In our notation, this condition is that for each arc $j$, the following function $G_j(x)$ is non-increasing.
\[
G_j(x) \equiv h_j x(1/F_j(x) - 1),
\]
where $F_j$ is the cdf of the hiding distribution on arc $j$. 

Here, we present examples of distributions that are forward biased but do not satisfy Kella's condition. We first consider the case of a balanced distribution on a two-arc star with unit length arcs, and cdfs $F(x) \equiv F_1(x) \equiv F_2(x)$ given by
\[
F(x) \equiv
\begin{cases}
2x/3 &\text{ if } 0 \le x < 1/2,\\
(1+2x)/6  &\text{ if } 1/2 \le x \le 1.
\end{cases} 
\]
\begin{figure}[h!]
  \centering
    \includegraphics[width=0.4\textwidth]{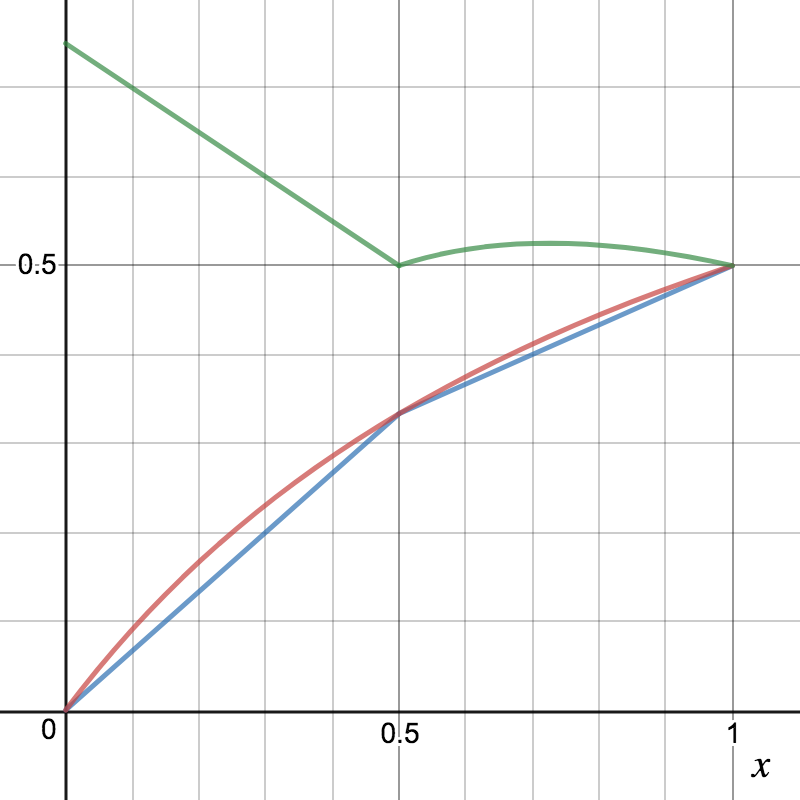}
      \caption{The cdf $F$ (bottom), $H(x)$ (middle) and $G(x)$ (top).}       
      \label{fig:kella-balanced}
\end{figure}

Figure~\ref{fig:kella-balanced} depicts this cdf, along with the function $H_1(x)\equiv H_2(x) = x/(1+x) \ge F(x)$, implying that this hiding distribution is forward biased. But the function $G(x) \equiv G_1(x) \equiv G_2(x)$, also depicted in the graph is not non-increasing for all $x \in [0,1]$, so $F(x)$ does not satisfy Kella's condition. In other words, Theorem~\ref{thm:star} implies that this hiding distribution is simply searchable, but Theorem 3.1 of \cite{Kella} does not.

For the case of hiding distributions that are not balanced, consider again the two-arc star, but this time with an atom of weight $2/3$ at the end of the left arc, and a distribution on the right arc given by the following cdf, $F_1$.
\[
F_1(x) \equiv 
\begin{cases}
x/2 &\text{ if } 0 \le x < 1/2, \\
(1+x)/6 &\text{ if } 1/2 \le x \le 1.
\end{cases}
\]
Then $H_1(x) = x/(1+x) \ge F_1(x)$, as depicted in Figure~\ref{fig:kella-unbalanced}, and clearly $H_2(x) \ge F_2(x)$, so by Theorem~\ref{thm:FB}, the hiding distribution is simply searchable. But the function $G_1(x)$ is not non-increasing for all $x$ in $[0,1]$, so $F_1(x)$ does not satisfy Kella's condition.
\begin{figure}[h]
  \centering
    \includegraphics[width=0.4\textwidth]{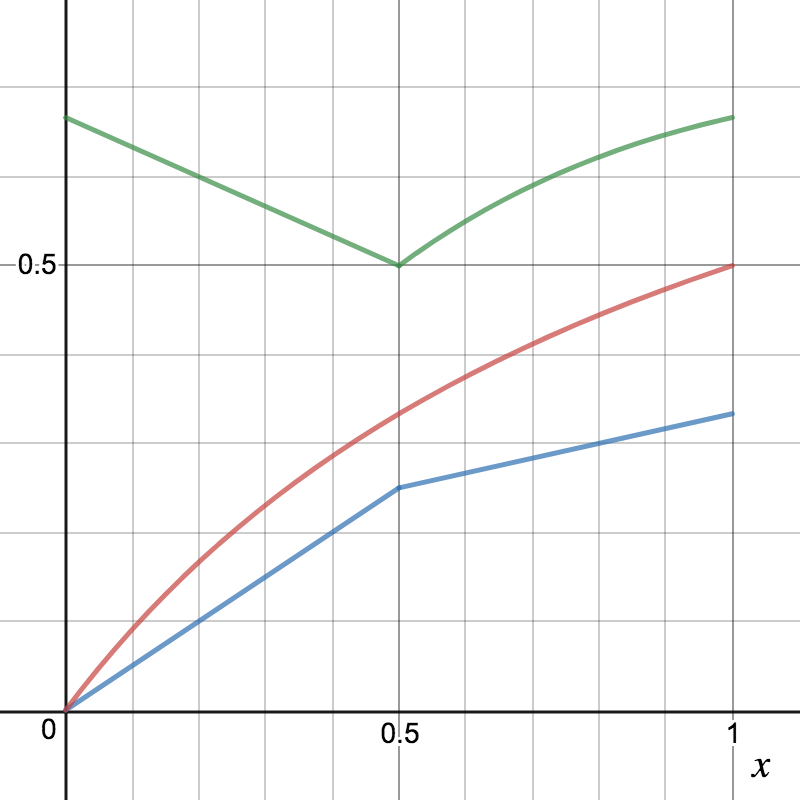}
      \caption{The cdf $F_1$ (bottom), $H_1(x)$ (middle) and $G_1(x)$ (top).}       
      \label{fig:kella-unbalanced}
\end{figure}

Of course Kella's condition applies to some unbalanced distributions on the $n$-star, $n>2$, which is not covered at all by our results.

\section{Conclusion} \label{sec:conclusion}

We have introduced a new class of hiding distributions on a tree called balanced, containing precisely those distributions for which every DF search has the same expected search time, for which we gave a simple formula. We then showed that for the subclass of monotone distributions, all DF searches are optimal. This includes the uniform distribution, which results in a simple method for choosing the point of the tree from which to begin the search that minimizes the optimal expected search time. We gave a concise characterization of the balanced hiding distributions on a star for which DF searches are optimal, and gave a necessary and sufficient condition for some DF search to be optimal on a two-arc star when the distribution may not be balanced. 

Further work could aim to specify the subclass of balanced distributions for which DF is optimal on a tree. More work is also needed to determine necessary and sufficient conditions for some DF search to be optimal when the hiding distribution is not balanced. It would be interesting to conduct further research on the problem of finding the optimal search for a target hidden according to a known distribution on an {\em arbitrary} network. One might also consider the problem of finding multiple targets hidden on a network according to a known distribution. A discrete version of this problem was considered in \cite{FLV18}, and a search game with multiple targets was solved in \cite{Lidbetter13}. Finally, these problems could all be generalized by considering {\em asymmetric} (or {\em windy}) networks, for which the time to traverse an arc depends on the direction of travel. Such networks have been widely studied in the context of the Traveling Salesman Problem, for example in~\cite{FT97} and~\cite{STV18}, and also in the context of search games in~\cite{Alpern10} and~\cite{AL14}.

\section*{Acknowledgements} Steve Alpern acknowledges support from the AFIT Graduate School of Engineering and Management, FA8075-14-D-0025.

This material is based upon work supported by the National Science Foundation under Grant No. CMMI-1935826.

\end{document}